\documentclass[12pt]{amsart}
\usepackage{fullpage,url,amssymb,enumerate,colonequals}
\usepackage[all,pdf]{xy} 

\usepackage{comment}
\usepackage{color}
\usepackage{charter,euler} 
\DeclareSymbolFont{bchoperators}{T1}{bch}{m}{n}
\SetSymbolFont{bchoperators}{bold}{T1}{bch}{b}{n}
\makeatletter
\renewcommand{\operator@font}{\mathgroup\symbchoperators}
\makeatother

\usepackage{titlesec} 
\titleformat{\section}{\normalfont\bfseries\filcenter}{\thesection}{1em}{}
\titleformat{\subsection}{\normalfont\bfseries\filcenter}{\thesubsection}{1em}{}

\usepackage[pdftex]{epsfig}
\newcommand{\Gr}[2]{\psfig{file=#1.pdf,width=#2}}

\newcommand{\C}{\mathbb{C}}

\newcommand{\HH}{\mathcal{H}}
\newcommand{\PP}{\mathbb{P}}
\newcommand{\Q}{\mathbb{Q}}
\newcommand{\R}{\mathbb{R}}
\newcommand{\Sph}{\mathbb{S}}
\newcommand{\Z}{\mathbb{Z}}
\newcommand{\calF}{\mathcal{F}}
\newcommand{\calM}{\mathcal{M}}
\newcommand{\calO}{\mathcal{O}}
\newcommand{\calT}{\mathcal{T}}
\newcommand{\calV}{\mathcal{V}}
\newcommand{\eps}{\varepsilon}

\newcommand{\GL}{\operatorname{GL}}
\newcommand{\PGL}{\operatorname{PGL}}
\newcommand{\SL}{\operatorname{SL}}
\newcommand{\SO}{\operatorname{SO}}
\newcommand{\SU}{\operatorname{SU}}
\renewcommand{\Re}{\operatorname{Re}}
\newcommand{\dist}{\operatorname{dist}}
\newcommand{\Aut}{\operatorname{Aut}}
\newcommand{\Hom}{\operatorname{Hom}}
\newcommand{\Res}{\operatorname{Res}}
\newcommand{\Mat}{\operatorname{Mat}}
\newcommand{\ordRes}{\operatorname{ordRes}}

\newcommand{\smm}[4]{\left(\begin{smallmatrix} #1 & #2 \\ #3 & #4 \end{smallmatrix}\right)}

\newcommand{\PBerk}{\PP^1_{\text{Berk}}} 
\newcommand{\HBerk}{\HH_{\text{Berk}}} 
\newcommand{\To}{\longrightarrow}

\numberwithin{equation}{section}

\newtheorem{theorem}{Theorem}[section]
\newtheorem{lemma}[theorem]{Lemma}
\newtheorem{corollary}[theorem]{Corollary}
\newtheorem{proposition}[theorem]{Proposition}

\theoremstyle{definition}
\newtheorem{definition}[theorem]{Definition}

\newtheorem{example}[theorem]{Example}

\newtheorem{algo}[theorem]{Algorithm}

\theoremstyle{remark}
\newtheorem{remark}[theorem]{Remark}

\definecolor{darkgreen}{rgb}{0,0.5,0}
\usepackage[
        colorlinks, citecolor=darkgreen,
        backref,
        pdfauthor={Benjamin Hutz, Michael Stoll}, 
        pdftitle={Smallest representatives of SL(2,Z)-orbits on binary forms and automorphisms of P^1},
]{hyperref}
\usepackage[alphabetic,backrefs,lite]{amsrefs} 

\begin{document}

\title{Smallest representatives of $\SL(2,\Z)$-orbits \\
       of binary forms and endomorphisms of $\PP^1$}

\author{Benjamin Hutz}
\address{Department of Mathematics and Statistics,
         Saint Louis University,
         St.~Louis, MO, USA}
\email{benjamin.hutz@slu.edu}

\author{Michael Stoll}
\address{Mathematisches Institut,
         Universit\"at Bayreuth,
         95440 Bayreuth, Germany.}
\email{Michael.Stoll@uni-bayreuth.de}
\urladdr{http://www.mathe2.uni-bayreuth.de/stoll/}

\date{\today}

\begin{abstract}
  We develop an algorithm that determines, for a given squarefree
  binary form~$F$ with real coefficients, a smallest representative
  of its orbit under~$\SL(2,\Z)$, either with respect to the Euclidean
  norm or with respect to the maximum norm of the coefficient vector.
  This is based on earlier work of Cremona and Stoll~\cite{Cremona2}.
  We then generalize our approach so that it also applies to the
  problem of finding an integral representative of smallest height in the
  $\PGL(2,\Q)$ conjugacy class of an endomorphism of the projective line.
  Having a small model of such an endomorphism is useful for various
  computations.
\end{abstract}


\maketitle


\section{Introduction}

Let $F = a_0 x^n + a_1 x^{n-1} y + \ldots + a_n y^n$ be a binary form with
real coefficients. We define its \emph{size} to be
\[ \|F\| = a_0^2 + a_1^2 + \ldots + a_n^2 \]
(this is the squared Euclidean norm of the coefficient vector)
and its \emph{height} to be the maximum norm
\[ H_\infty(F) = \max \{|a_0|, |a_1|, \ldots, |a_n|\} \,. \]
If $F$ has coefficients in~$\Z$
with $\gcd(a_0,\ldots,a_n) = 1$, then $H_\infty(F) = H(F)$
is the (multiplicative) global height of~$F$ in the sense that
it is the global height of the coefficient vector of~$F$, considered as a point
in projective space~$\PP^n$. If $F$ has coefficients in~$\Q$,
then $H(F) = H_0(F) H_\infty(F)$ with
\[ H_0(F) = \prod_{p \text{\ prime}} \max\{|a_0|_p, |a_1|_p, \ldots, |a_n|_p\} \,. \]
Since $H_0(F)$ does not change under the action of~$\SL(2,\Z)$, we
can use~$H_\infty(F)$ as a proxy for the global height~$H(F)$ for our purposes.

Our goal in this note will be to find, for a given~$F$
(without multiple factors, say), a smallest representative~$F_0$ in its
$\SL(2,\Z)$-orbit, in the sense that $\|F_0\|$ is minimal among all forms
in the orbit of~$F$, or in the sense that $H_\infty(F_0)$ (or equivalently,
$H(F_0)$ when $F_0$ has coefficients in~$\Q$) is minimal within the orbit.

More generally, we may want to consider some kind of geometric object~$\Phi$
related to the projective line (over~$\Q$, say), given by some polynomials
with respect to some chosen coordinates on~$\PP^1$. Then we usually can associate
to~$\Phi$ (in a natural, i.e., coordinate-independent way) a finite
set of points on~$\PP^1$, or equivalently, a binary form $F = F(\Phi)$.
``Natural'' means that this association is compatible with
the action of $\PGL(2) = \Aut(\PP^1)$ on both sides. In this way, we
can set up a ``reduction theory'' for the objects~$\Phi$ by selecting
a suitable ``reduced'' representative $F \cdot \gamma$ in the $\SL(2,\Z)$-orbit
of~$F$ and declaring $\Phi \cdot \gamma$ to be the reduced representative
in the $\SL(2,\Z)$-orbit of~$\Phi$. This is the approach taken
in~\cite{Stoll2011b} in the setting of~$\PP^n$ for general~$n$.
This works reasonably well when we just want to have a way of selecting
a canonical representative of moderate size. If we want to find a
representative of smallest height, then more work is required:
we have to relate the height of~$\Phi$ to the size of~$F(\Phi)$
and determine a bound on the distance we can move from the canonical representative
 without increasing the height. The application we have
in mind is to endomorphisms of~$\PP^1$; this is discussed in some detail
in Section~\ref{S:endo} below.

\subsection*{Main Results}

Our main results are as follows.
\begin{enumerate}\addtolength{\itemsep}{1mm}
  \item We prove a result (Theorem~\ref{T:bound}) that bounds the size of a
        binary form~$F$ in terms of its Julia invariant~$\theta(F)$ (see below)
        and the hyperbolic distance of its covariant~$z(F)$ (see below) to
        the ``center''~$i$ of the hyperbolic plane~$\HH$.
  \item Based on this result, we construct an algorithm (Algorithm~\ref{Algo:main})
        that determines a representative of smallest size in the $\SL(2,\Z)$-orbit
        of a given binary form.
  \item We generalize this algorithm so that it can be used with different
        notions of ``size'' and for more general objects associated to~$\PP^1$.
        We apply this specifically to automorphisms of~$\PP^1$; see Section~\ref{sect_ds}.
  \item In the latter context, we also show how to find minimal representatives
        up to $\GL(2,\Q)$-conjugacy. This involves the determination of all
        $\GL(2,\Z)$-orbits of minimal models; see the algorithms in Section~\ref{S:orbits}.
        In this context, we prove that an automorphism of even degree has
        only one orbit of minimal models (Proposition~\ref{prop:even minimal}),
        but there can be an arbitrary number of orbits for odd degree
        (Proposition~\ref{prop:odd minimal}).
\end{enumerate}

\subsection*{Structure of the paper}

The paper proceeds as follows. In Section~\ref{S:endo} we review the theory
of endomorphisms of~$\PP^1$ and in Section~\ref{S:red}, we recall the reduction theory
for binary forms as developed in~\cite{Cremona2}.
This provides us with the canonical representative of an $\SL(2,\Z)$-orbit
of (real) binary forms, as follows. We associate to a binary form~$F$
a point~$z(F)$ in the upper half-plane; this association is $\SL(2,\R)$-equivariant.
Then $F$ is the canonical representative if $z(F)$ is in the standard
fundamental domain for the action of~$\SL(2,\Z)$ on the upper half-plane.
The main result in Cremona-Stoll~\cite{Cremona2} is that $z(F)$ is well-defined
and that the canonical representative should correspond to a model of~$F$ with small size.
However, this canonical representative may not be the model with \emph{smallest} size,
but the expectation in their work is that it should be near, in terms of hyperbolic
distance of~$z(F)$, to the smallest model.

Section~\ref{sect_bound} is the heart of the paper.
Here we make explicit the relation of the size~$\|F\|$ of a binary form~$F$ to its
``Julia invariant''~$\theta(F)$ and the hyperbolic distance of~$z(F)$
to the point~$i$ (note that we will denote this point mostly by~$j$ instead).
Specifically, in Theorem \ref{T:bound} we prove an explicitly computable bound on how far,
in terms of the hyperbolic distance, the covariant~$z(F)$ of the smallest representative
can be from the point~$i$. This allows us to solve the problem of minimizing the size with
respect to the Euclidean norm within an $\SL(2,\Z)$-orbit of binary forms by reducing the
problem to a finite search space of elements of~$\SL(2,\Z)$. We can
enumerate them and find the element that gives the smallest representative.
For this method to be practical, we need the bound on the distance to be of reasonable size.
The examples included in this paper demonstrate that the presented bounds do result
in a practical algorithm. Since this size and the height are fairly
closely related, we also obtain a similar way of minimizing the height.
Section~\ref{S:algo} spells out the resulting algorithm in some detail
and gives an example.

Section~\ref{sect_ds} then applies
this method to the problem of finding a model of minimal height of a given
endomorphism of~$\PP^1$. This problem has not been addressed before in the literature
and has practical applications for other algorithms related to dynamical systems as
discussed in Section~\ref{S:endo}. The explicitly computable distance bound in this case
is given in Corollary~\ref{C:endobound}. This method is demonstrated with another example.
Finally, in Section~\ref{S:orbits}, we answer some questions concerning the number
of distinct $\GL(2,\Z)$-orbits of integral models with the same (minimal) resultant,
posed by Bruin and~Molnar~\cite{Bruin3}. In particular, Bruin and Molnar gave an example
in odd degree of an endomorphism with multiple distinct $\GL(2,\Z)$-orbits of minimal models
and asked if an even degree example were possible. Since the algorithm of Section~\ref{sect_ds}
finds a smallest model in a single $\SL(2,\Z)$-orbit, we must have a way to check all distinct
$\SL(2,\Z)$-orbits to determine a smallest representative of the entire conjugacy class.
We prove that any even degree endomorphism has a single $\GL(2,\Z)$-orbit of minimal models
(Proposition~\ref{prop:even minimal}) and extend the Bruin-Molnar example to demonstrate
odd degree endomorphisms with arbitrarily many distinct orbits of minimal models
(Proposition~\ref{prop:odd minimal}). In addition, we provide a simple algorithm to compute
a representative from all $\GL(2,\Z)$-orbits of minimal models of a given endomorphism of~$\PP^1$,
extending the algorithm from Section~\ref{sect_ds} to finding a minimal representative over
the entire conjugacy class of the endomorphism.

\noindent\textbf{Acknowledgments.}
We thank the anonymous referee for some helpful comments.


\section{Endomorphisms of the projective line} \label{S:endo}

Let $\Hom_d$ be the space of degree~$d$ endomorphisms of~$\PP^1$.
There is a natural action of~$\PGL(2)$ by conjugation on~$\Hom_d$.
We define the quotient space for this action as $\calM_d = \Hom_d/\PGL(2)$.
Levy~\cite{Levy} proved that $\calM_d$ exists as a geometric quotient,
which we call the \emph{moduli space of dynamical systems of degree~$d$
morphisms of~$\PP^1$}.
For $f \in \Hom_d$, we denote the corresponding element of the moduli space
as $[f] \in \calM_d$. After choosing coordinates, we can write $f$ as a pair
of degree~$d$ homogeneous polynomials, which we call a \emph{model} for~$[f]$.
Given $\alpha \in \PGL(2)$, we denote the conjugate as
$f^{\alpha} = \alpha^{-1} \circ f \circ \alpha$.
When working over an infinite field, there are infinitely many models for
any conjugacy class. The choice of a ``good'' model can affect a variety
of properties and algorithms.
For example, the \emph{field of definition} of a model~$f$ is the smallest
field containing the coefficients of~$f$. Different models of $[f] \in \calM_d$
may have different fields of definition. The fields of definition are studied
in~\cites{Hutz11,Silverman12}.
Here we are considering models defined over~$\Q$. In this context,
an \emph{integral model} of~$[f]$ is a model of~$f$ consisting of
a pair of polynomials with coefficients in~$\Z$. Any model over~$\Q$
can be scaled to give an integral model with the property that the coefficients
of the two polynomials are coprime.

Given an integral model~$f$,
we define the \emph{resultant}~$\Res(f)$ to be the resultant
of its two defining polynomials considered both as having degree $d$.
When working over a global field, reducing
modulo primes can yield information on the arithmetic dynamics of~$f$.
In particular, this local-global transfer of information is a key piece
of an algorithm to determine all rational preperiodic points for a given
map~\cite{Hutz12}. However, whether reduction modulo the prime commutes
with iteration affects what local information can be obtained.
A \emph{prime of good reduction} is a prime~$\mathfrak{p}$ such that reduction
modulo~$\mathfrak{p}$ commutes with iteration.
For $f \in \Hom_d$, the primes of good reduction are the primes that do
not divide the resultant. Consequently, the problem of finding an integral model
of $[f] \in \calM_d$ with smallest (norm of the) resultant is important.
Such a model is called
a \emph{minimal model for~$[f]$}, and an algorithm to determine a minimal model
is given by Bruin and Molnar~\cite{Bruin3}. This algorithm is implemented in the
computer algebra system Sage~\cite{sage}. It is important to note that minimal
models are, in general, not unique.
For example, conjugating by an element of~$\SL(2, \Z)$ leaves the resultant
unchanged. Consequently, there are infinitely many minimal models
for a given~$[f]$.
Furthermore, we prove in Proposition~\ref{prop:odd minimal}
that there are maps with arbitrarily many distinct $\SL(2,\Z)$-orbits of minimal models.
Choosing a ``best'' minimal model depends on the application in mind.
For example, when working with polynomials, it is conventional to move the totally
ramified fixed point to the point at infinity. Similarly, there are models that
are convenient for working with critical points~\cite{Ingram2}
or multipliers~\cite{Milnor}. The focus of this article is a minimal model of minimal
height.
For a model $f = [F : G] \in \Hom_d$ with homogeneous polynomials $F$ and~$G$,
we define the \emph{height of~$f$} as
\[ H(f) = H(F,G)\,, \]
i.e., the height of the concatenated coefficient vectors. Note that this height
does not change if we replace $[F : G]$ by~$[\lambda F : \lambda G]$.

There are a number of properties and algorithms where bounds are dependent
on~$H(f)$. For example, in the algorithm to determine all rational preperiodic
points~\cite{Hutz12} an upper bound on the height of a rational preperiodic
point is determined that depends on~$H(f)$.
Furthermore, the minimal height over the models with minimal resultant
defines a height on the moduli space~$\calM_d$.
See~\cite{Silverman10}*{Conjecture~4.98} for a dynamical version of Lang's height
conjecture related to heights on $\calM_d$.

\begin{definition}
  Given $[f] \in \calM_d$, we call $g$ a \emph{reduced model of $[f]$} if $g$
  is a minimal model for~$[f]$ with smallest height~$H(g)$.
\end{definition}

We will use the ideas of~\cite{Cremona2} together with the new bounds
from Section~\ref{sect_bound} to devise an algorithm that finds a model
of smallest height in the $\SL(2,\Z)$-orbit of a given model. Together
with an extension of the algorithm from~\cite{Bruin3} that finds
a representative in each $\GL(2,\Z)$-orbit of minimal models of a given~$[f]$,
this results in a procedure that produces a reduced model of any given~$[f] \in \calM_d$.
Note that the minimal height in the $\GL(2,\Z)$-orbit of a given model
is the same as the minimal height in its $\SL(2,\Z)$-orbit, so it is sufficient
to apply our reduction algorithm to a representative of each $\GL(2,\Z)$-orbit
of minimal models of~$[f]$.


\section{Reduction of binary forms} \label{S:red}

We recall the reduction theory of binary forms, as described in~\cite{Cremona2}
(following earlier work of Hermite~\cite{Hermite} and Julia~\cite{Julia}).
In the following, the degree~$n$ of the form is always assumed to be at least~$3$.
The space of binary forms of degree~$n$ over a ring~$R$ is denoted~$R[x,y]_n$.

We take a more general approach than in the introduction
and consider a binary form with complex (instead of real) coefficients
\[ F = a_0 x^n + a_1 x^{n-1} y + \ldots + a_n y^n \in \C[x,y]_n \,. \]
In this context, the \emph{size}~$\|F\|$ of~$F$ is defined as
\[ \|F\| = |a_0|^2 + |a_1|^2 + \ldots + |a_n|^2
         = \int_0^1 \left|F(e^{2\pi i \phi}, 1)\right|^2\,d\phi \,.
\]
The groups $\GL(2,\C)$ and~$\SL(2,\C)$ act on the space~$\C[x,y]_n$ of
binary forms of degree~$n$ via linear substitution of the variables;
this is an action on the right. Concretely,
\[ F(x,y) \cdot \begin{pmatrix} a & b \\ c & d \end{pmatrix} = F(ax+by, cx+dy) \,. \]
We write $\HH_3 = \C \times \R_{>0}$ for three-dimensional hyperbolic space in the upper
half-space model; we write $t + uj$ for
the point $(t,u) \in \C \times \R_{>0} = \HH_3$. We identify the hyperbolic
upper half-plane~$\HH$ with the subset $\R \times \R_{>0}$ of~$\HH_3$.
There is a standard left action of~$\SL(2,\C)$ on~$\HH_3$ that restricts
to the standard action of~$\SL(2,\R)$ on~$\HH$ via M\"obius transformations.
The standard fundamental domain of the action of~$\SL(2,\Z)$ on~$\HH$ is
\[ \calF = \{t + uj \in \HH : |t| \le \tfrac{1}{2}, t^2 + u^2 \ge 1\} \,. \]

The main idea followed in~\cite{Cremona2} is to set up a map
\[ z \colon \C[x,y]'_n \To \HH_3 \]
(where $\C[x,y]'_n$ is a suitable subset of~$\C[x,y]_n$ that contains
the squarefree forms) that is covariant with respect to the $\SL(2,\C)$-actions
on both sides, in the sense that
\[ z(F \cdot \gamma) = \gamma^{-1} \cdot z(F) \qquad
   \text{for all $F \in \C[x,y]'_n$ and all $\gamma \in \SL(2,\C)$.}
\]
We also require $z$ to be compatible with complex conjugation in the
sense that if $z(F) = t + uj$, then $z(F^c) = \bar{t} + uj$,
where $F^c$ denotes the polynomial obtained from~$F$
by replacing each coefficient with its complex conjugate.
This ensures that $z(F) \in \HH$ when $F$ has real coefficients.
For such~$F$ we then say that $F$ is \emph{reduced} when $z(F) \in \calF$.
Since $\calF$ is a fundamental domain for the $\SL(2,\Z)$-action,
there will be a reduced representative~$F_0$ in each $\SL(2,\Z)$-orbit.
If $z(F_0)$ is in the interior of~$\calF$, then $F_0$ is uniquely determined
(up to sign when $n$ is odd, since $-I_2 \in \SL(2,\Z)$ acts trivially on~$\HH$).
When $z(F_0)$ is on the boundary of~$\calF$, there is some ambiguity,
which can be resolved by removing part of the boundary. However, for practical
purposes, this ambiguity is usually not a problem.
We also note that it is impossible to find out whether a numerically
computed point is exactly on the boundary.

There are many choices for this map~$z$ when $n \ge 5$ (for $n = 3$ or~$4$
there is only one choice, which is forced by the symmetries of cubics
and quartics; see~\cite{Cremona2}*{Prop.~3.4}). The choice made
by Julia~\cite{Julia} and in~\cite{Cremona2} can be described in the following way.
We write $F \in \C[x,y]_n$ as
\[ F(x,y) = \prod_{k=1}^n (\beta_k x - \alpha_k y) \]
with $\alpha_k, \beta_k \in \C$. Then we define, for $t+uj \in \HH_3$,
\[ R(F, t+uj) = \prod_{k=1}^n \frac{|\alpha_k - \beta_k t|^2 + |\beta_k|^2 u^2}{u} \in \R_{\ge 0} \,; \]
we note that $R(F, t+uj) > 0$ when $F \neq 0$.
It can be checked that $R$ is invariant under the $\SL(2,\C)$-action:
$R(F \cdot \gamma, z) = R(F, \gamma \cdot z)$.

Now we take $\C[x,y]'_n \subset \C[x,y]_n$ to be the subset of
stable forms, where a form~$F$ is \emph{stable} when no linear factor of~$F$
has multiplicity $\ge n/2$. It is shown in~\cite{Cremona2}*{Prop.~5.1}
that for $F \in \C[x,y]'_n$,
the function $\HH_3 \to \R_{>0}$, $z \mapsto R(F, z)$, has a unique
minimizer~$z(F)$; we define the \emph{Julia invariant} of~$F$ to be
the minimal value,
\[ \theta(F) = R(F, z(F)) > 0 \,. \]
There is a nice geometric interpretation of~$z(F)$. Namely, $z(F)$ is
the unique point in~$\HH_3$ with the property that the sum of the unit
tangent vectors pointing in the direction of the $n$~roots of~$F$
(in~$\PP^1(\C)$, considered as the ideal boundary of~$\HH_3$) vanishes;
see~\cite{Cremona2}*{Cor.~5.4}.


\section{Bounds for the size of a binary form} \label{sect_bound}

Our goal in this section is to relate $\theta(F)$ and the size~$\|F\|$
of a form $F \in \C[x,y]'_n$.
The first step is a comparison between $\|F\|$ and~$R(F, j)$.
Note that for $F = \prod_{k=1}^n (\beta_k x - \alpha_k y)$ we have that
\[ R(F, j) = \prod_{k=1}^n (|\alpha_k|^2 + |\beta_k|^2) \,. \]

\begin{proposition} \label{P:bound1}
  Let $F \in \C[x,y]_n$. Then
  \[ 2^{1-n} R(F, j) \le \|F\| \le 2^{-n} \binom{2n}{n} R(F,j) \,. \]
  If $F \in \C[x,y]'_n$ with $z(F) = j$, then $\|F\| \le R(F,j)$.
\end{proposition}

\begin{proof}
  We begin with an observation related to the size of~$F$.
  Define $\tilde{F}(x,y) = F^c(y,x)$; then clearly
  $\|F\|$ is the coefficient of~$x^n y^n$ in~$F(x,y) \tilde{F}(x,y)$.
  Now assume that $F = G H$. Then
  \[ F \tilde{F} = G H \tilde{G} \tilde{H} = G \tilde{H} \cdot \widetilde{G \tilde{H}} \,, \]
  which implies that $\|F\| = \|GH\| = \|G \tilde{H}\|$.
  It is clear from the definition that
  \[ R(GH, j) = R(G, j) \cdot R(H, j) = R(G, j) \cdot R(\tilde{H}, j) = R(G\tilde{H}, j) \,. \]
  Writing $F = \prod_{k=1}^n (\beta_k x - \alpha_k y)$, this allows
  us to replace each factor $\beta_k x - \alpha_k y$ with $|\alpha_k| > |\beta_k|$
  by its reverse conjugate $-\bar{\alpha}_k x + \bar{\beta}_k y$.
  Since scaling~$F$ by a constant~$\lambda$ clearly scales both $\|F\|$
  and~$R(F, j)$ by~$|\lambda|^2$, we can in this way assume without loss of
  generality that
  \[ F = (x - \alpha_1 y) \cdots (x - \alpha_n y)
     \qquad\text{with \quad $|\alpha_1|, \ldots, |\alpha_n| \le 1$.}
  \]

  We now show the first inequality. We have that
  \[ \|F\| \ge 1 + \prod_{k=1}^n |\alpha_k|^2 \qquad\text{and}\qquad
     2^{1-n} R(F, j) = 2^{1-n} \prod_{k=1}^n (1 + |\alpha_k|^2) \,.
  \]
  Write $x_k = |\alpha_k|^2 \in [0,1]$ and observe that the difference of
  $2^{1-n} (1+x_1) \cdots (1+x_n)$ and $1 + x_1 \cdots x_n$ is an affine-linear
  function in each of the~$x_k$. This implies that the extrema of this
  difference must occur at some vertices of the unit cube~$[0,1]^n$.
  But it is easy to see that
  \[ 1 + x_1 \cdots x_n \ge 2^{1-n} (1+x_1) \cdots (1+x_n) \]
  whenever $x_1, \ldots, x_n \in \{0,1\}$. This proves the first inequality.

  We now turn to the second inequality. We write $e(\phi) = e^{2\pi i \phi}$
  for $\phi \in \R$. Fix $\alpha \in \C$ and set $\beta = 2 \alpha/(1 + |\alpha|^2)$.
  Then
  \[ |e(\phi) - \alpha|^2
       = \frac{1 + |\alpha|^2}{2} \bigl(2 - \beta e(-\phi) - \bar{\beta} e(\phi)\bigr) \,.
  \]
  So with $F$ as above (note, though, that we do not need to assume that
  all roots have absolute value $\le 1$), we have that
  \begin{align*}
     \|F\| &= \int_0^1 |F(e(\phi),1)|^2\,d\phi\\
           &= \prod_{k=1}^n \frac{1+|\alpha_k|^2}{2}
             \int_0^1 \prod_{k=1}^n \bigl(2 - \beta_k e(-\phi) - \bar{\beta}_k e(\phi)\bigr)\,d\phi\\
           &= 2^{-n} R(F, j)
             \int_0^1 \prod_{k=1}^n \bigl(2 - \beta_k e(-\phi) - \bar{\beta}_k e(\phi)\bigr)\,d\phi
  \end{align*}
  with $\beta_k = 2 \alpha_k/(1 + |\alpha_k|^2)$. Note that each factor
  in the product is a real number in~$[0,4]$:
  \[ 2 - \beta_k e(-\phi) - \bar{\beta}_k e(\phi) = 2 - 2 \Re(\beta_k e(-\phi)) \]
  and $|\beta_k| \le 1$. So we can apply the AGM inequality to the product.
  This gives
  \[ \|F\| \le 2^{-n} R(F, j) \int_0^1 \bigl(2 - \beta e(-\phi) - \bar{\beta} e(\phi)\bigr)^n \,d\phi \,, \]
  where $\beta = \sum_{k=1}^n \beta_k/n$ is the arithmetic mean of the~$\beta_k$.
  Now one can check that $\beta_k$ is the vertical projection to the unit disk
  of~$\alpha_k$, viewed as a point on the Riemann sphere, which forms the ideal
  boundary of the Poincar\'e ball model of~$\HH_3$. When $z(F) = j$, then
  the sum of the points on the Riemann sphere corresponding to the~$\alpha_k$
  vanishes by the geometric characterization of~$z(F)$, so in this case,
  we have that $\beta = 0$, which implies $\|F\| \le R(F, j)$.
  In the general case, expanding the product and integrating gives
  \[ \int_0^1 \bigl(2 - \beta e(-\phi) - \bar{\beta} e(\phi)\bigr)^n \,d\phi
       = \sum_{\ell=0}^\infty 2^{n-2\ell} \binom{n}{2\ell} \binom{2\ell}{\ell} |\beta|^{2\ell} \,,
  \]
  which shows that the maximum is attained when $|\beta|$ is maximal,
  so for $|\beta| = 1$. It is clear that the integral does not depend
  on the argument of~$\beta$, so we can take $\beta = -1$. Then the
  integrand simplifies to
  \[ \bigl(2 + e(\phi) + e(-\phi)\bigr)^n
      = \bigl(e(\tfrac{1}{2}\phi) + e(-\tfrac{1}{2}\phi)\bigr)^{2n}
  \]
  and so the value of the integral is the constant term of $(X + X^{-1})^{2n}$, which
  is~$\binom{2n}{n}$. This proves the second inequality in the general case.
\end{proof}

We note that both inequalities are sharp in the general case:
the lower bound is attained for $F = x^n - y^n$ (which has $z(F) = j$),
and the upper bound is
attained for $F = (x-y)^n$. The upper bound in the case $z(F) = j$
is sharp for $n$~even, in the sense that
\[ \sup_{F \colon z(F) = j} \frac{\|F\|}{R(F, j)} = 1 \,; \]
this is attained as $F$ gets close to $(xy)^{n/2}$. For odd~$n$,
we cannot balance the roots in this way while they tend to zero or~infinity,
so the bound will not be sharp, but it will not be too far off when $n$
is large.

From now on, we assume that $F$ is stable, i.e., $F \in \C[x,y]'_n$;
then $z(F)$ and~$\theta(F)$ are defined.

We want to relate $\|F\|$ and~$\theta(F)$. Proposition~\ref{P:bound1}
tells us that
\[ 2^{1-n} \le \frac{\|F\|}{\theta(F)} \le 1 \]
when $z(F) = j$, since then $\theta(F) = R(F, z(F)) = R(F, j)$.
This leads us to expect that $\|F\|$ can be bounded in terms
of~$\theta(F)$ and the (hyperbolic) distance between $z(F)$ and~$j$:
we have that
\[ \frac{\|F\|}{\theta(F)}
     = \frac{\|F\|}{R(F, z(F))}
     = \frac{\|F\|}{R(F, j)} \cdot \frac{R(F, j)}{R(F, z(F))} \,;
\]
the first factor is bounded above and below by Proposition~\ref{P:bound1},
and since $R(F, z)$ attains its minimum when $z = z(F)$, the second factor
should grow when $z(F)$ moves away from~$j$.

Let $\gamma \in \SL(2,\C)$ be such that $F_0 = F \cdot \gamma^{-1}$
satisfies $z(F_0) = j = \gamma \cdot z(F)$.
By the invariance of~$R$, we then have that
\[ \frac{R(F, j)}{R(F, z(F))}
     = \frac{R(F_0 \cdot \gamma, j)}{R(F_0 \cdot \gamma, z(F))}
     = \frac{R(F_0, \gamma \cdot j)}{R(F_0, \gamma \cdot z(F))}
     = \frac{R(F_0, \gamma \cdot j)}{R(F_0, j)} \,.
\]
If $\dist(z,z')$ denotes hyperbolic distance in~$\HH_3$, then,
since distance is invariant under the $\SL(2,\C)$-action,
$\dist(z(F), j) = \dist(\gamma \cdot j, j)$. So it is enough to
bound $R(F_0, z)/R(F_0, j)$ in terms of $\dist(z, j)$ when $z(F_0) = j$.

Now for $z = t + uj$, the distance to~$j$ is given by
\begin{equation} \label{E:coshdelta}
  \cosh \dist(z, j) = \frac{|t|^2 + u^2 + 1}{2u} \,.
\end{equation}
In particular, $\dist(e^\delta j, j) = |\delta|$.
We can assume that $F_0 = \prod_{k=1}^n (x - \alpha_k y)$. Let $\phi_k \in \Sph^2$
be the point on the Riemann sphere that corresponds to~$\alpha_k$
under stereographic projection (in coordinates in $\R^3 = \C \times \R$,
$\phi_k = \left(\frac{2\alpha_k}{|\alpha_k|^2+1},
                \frac{|\alpha_k|^2-1}{|\alpha_k|^2+1}\right)$;
the first component is~$\beta_k$ in the proof of Proposition~\ref{P:bound1}).
The condition $z(F_0) = j$ then is equivalent to $\sum_k \phi_k = 0$.
We derive a formula for the quotient $R(F_0,z)/R(F_0,j)$.

\begin{lemma} \label{L:Rquot}
  Let $0 \neq F_0 \in \C[x,y]_n$ and $z \in \HH_3$. We set $\delta = \dist(z, j)$
  and denote the unit tangent vector at~$j$ in the direction of~$z$ by~$\phi \in \Sph^2$
  ($\phi$ is arbitrary when $z = j$). Let $\phi_k \in \Sph^2$ be as in the preceding paragraph.
  Then
  \[ \frac{R(F_0, z)}{R(F_0, j)}
       = \prod_{k=1}^n (\cosh \delta + \langle \phi, \phi_k \rangle \sinh \delta) \,,
  \]
  where $\langle {\cdot}, {\cdot} \rangle$ denotes the standard
  inner product on~$\R^3$ (recall that $\phi, \phi_k \in \Sph^2 \subset \R^3$).
\end{lemma}

\begin{proof}
  Both sides do not change when we replace $F_0$ and~$z$
  by $F_0 \cdot \gamma$ and $\gamma^{-1} \cdot z$, respectively, for
  some $\gamma \in \SU(2)$. This allows us to assume that $\phi$ points upward;
  then $z = e^\delta j$ and $\langle \phi, \phi_k \rangle = (|\alpha_k|^2-1)/(|\alpha_k|^2+1)$.
  By definition of~$R$ and with $z = e^\delta j$, we obtain that
  \begin{align*}
    \frac{R(F_0, z)}{R(F_0, j)}
      &= \prod_{k=1}^n \frac{|\alpha_k|^2 e^{-\delta} + e^{\delta}}{|\alpha_k|^2 + 1} \\
      &= \prod_{k=1}^n \left(\frac{e^\delta + e^{-\delta}}{2}
                               + \frac{|\alpha_k|^2 - 1}{|\alpha_k|^2 + 1}
                                    \cdot \frac{e^\delta - e^{-\delta}}{2}\right) \\
      &= \prod_{k=1}^n (\cosh \delta + \langle \phi, \phi_k \rangle \sinh \delta) \,. \qedhere
  \end{align*}
\end{proof}

We can use this to deduce bounds for the $R$-quotient. We will derive
an upper bound assuming that $z(F_0) = j$. Regarding a lower bound,
note that the factor
$(\cosh \delta + \langle \phi, \phi_k \rangle \sinh \delta)$
in the product above can be as small as~$e^{-\delta}$ when $\phi_k = -\phi$.
So if lots of roots are in the direction opposite to~$z$, then the product
can get quite small. To get a reasonable bound, we assume that $F_0$ is
squarefree. Then the directions~$\phi_k$ cannot get too close to one
another, and so at most one factor can get really small. This approach
leads to the lower bound below.

\begin{proposition} \label{P:eps}
  Let $F_0 \in \C[x,y]'_n$ be squarefree and such that $z(F_0) = j$.
  There is a constant $\eps(F_0) > 0$ such that for all $z \in \HH_3$, we have that
  \[ \eps(F_0) (\cosh \delta)^{n-2} \le \frac{R(F_0, z)}{R(F_0, j)} \le (\cosh \delta)^n \,, \]
  where $\delta = \dist(z, j)$.
\end{proposition}

\begin{proof}
  By Lemma~\ref{L:Rquot},
  \[ \frac{R(F_0, z)}{R(F_0, j)}
       = \prod_{k=1}^n (\cosh \delta + \langle \phi, \phi_k \rangle \sinh \delta) \,.
  \]
  Since $z(F_0) = j$, we have that $\sum_{k=1}^n \phi_k = 0$.
  Therefore, by the AGM inequality,
  \[ \prod_{k=1}^n (\cosh \delta + \langle \phi, \phi_k \rangle \sinh \delta)
      \le \bigl(\cosh \delta
        + \bigl\langle \phi, \frac{1}{n} \sum_{k=1}^n \phi_k \bigr\rangle \sinh \delta\bigr)^n
      = (\cosh \delta)^n \,.
  \]
  By assumption, the $\phi_k$ are pairwise distinct, which implies that
  \[ \eta = 1 - \max_{k \neq k'} \sqrt{\frac{\langle \phi_k, \phi_{k'} \rangle + 1}{2}}
          > 0 \,.
  \]
  Then for any~$\phi \in \Sph^2$, it follows that $\langle \phi, \phi_k \rangle > 1 - \eta$
  for at most one~$k$, say $k_0$ ($1-\eta$ is the cosine of half the angle between the
  closest two~$\phi_k$). Applying this to~$-\phi$, we see that for all $k \neq k_0$,
  \[ \cosh \delta + \langle \phi, \phi_k \rangle \sinh \delta
       \ge \cosh \delta + (\eta - 1) \sinh \delta
       = e^{-\delta} + \eta \sinh \delta
       \ge \eta \cosh \delta \,.
  \]
  For $k_0$ we have that
  \[ \cosh(\delta) - \langle \phi, \phi_{k_0} \rangle \sinh(\delta)
            \ge e^{-\delta} \ge \frac{1}{2\cosh(\delta)} \,.
  \]
  This gives that
  \[ \frac{R(F_0, z)}{R(F_0, j)}
       \ge \frac{\eta^{n-1}}{2} (\cosh \delta)^{n-2} \,.
  \]
  So the lower bound holds with $\eps(F_0) = \eta^{n-1}/2$.
\end{proof}

It is fairly clear that the value given for~$\eps(F_0)$ in the proof is
unlikely to be optimal. Writing $\rho = \tanh \delta$, the optimal value
we can take is
\begin{equation} \label{E:eps}
  \eps(F_0) = \inf_{\phi \in \Sph^2, 0 \le \rho < 1}
                 \frac{\prod_{k=1}^n (1 + \langle \phi, \phi_k \rangle \rho)}{1 - \rho^2} \,.
\end{equation}
When $n = 3$, the only way for $\phi_1 + \phi_2 + \phi_3$
to vanish is that the $\phi_k$ are the vertices of an equilateral triangle.
Without loss of generality, we can take them to be
$1, \omega, \omega^2 \in \Sph^1 = \Sph^2 \cap (\C \times \{0\})$, where $\omega = e(1/3)$.
If the projection of~$\phi$ to the complex plane is $\lambda e(\varphi)$,
with $0 \le \lambda \le 1$, then the expression under the infimum works out as
\[ \frac{1 - \tfrac{3}{4} \lambda^2 \rho^2 + \tfrac{1}{4} \cos (6\pi\varphi) \lambda^3 \rho^3}%
        {1 - \rho^2} \,.
\]
It is clear that, for fixed~$\rho$, the expression will be minimal when $\cos (6\pi\varphi) = -1$
and $\lambda = 1$. The expression then simplifies to
\[ \frac{(1 + \tfrac{1}{2}\rho)^2}{1 + \rho} = 1 + \frac{\rho^2}{4(1+\rho)} \,, \]
which is minimal for $\rho = 0$. This gives the following.

\begin{lemma} \label{L:n=3}
  If $F_0 \in \C[x,y]'_3$ satisfies $z(F_0) = j$, then we can take $\eps(F_0) = 1$.
\end{lemma}

If we would use the value $\eta^2/2$ from the proof of Proposition~\ref{P:eps}, then we
would obtain $\eps(F_0) = \tfrac{1}{8}$ instead. In general, we can at least
in principle compute (an approximation to) the optimal~$\eps(F_0)$ by solving
the optimization problem in~\eqref{E:eps}. We remark here that when working
over~$\R$ instead of~$\C$, we can restrict $\phi$ to run over
the circle $\Sph^2 \cap (\R \times \R)$ (inside $\R^3 = \C \times \R$).
This simplifies the computation and can also result in a better bound.
A further improvement (which can also be used to extend the applicability
of the resulting algorithm from squarefree forms to stable forms) is based
on the following result.

\begin{lemma} \label{L:epsdelta}
  Let $F_0 \in \C[x,y]'_n$ be stable with $z(F_0) = j$. Then
  \[ \eps_{F_0} \colon \R_{\ge 0} \To \R_{\ge 1}\,, \qquad
                       \delta \longmapsto \min_{z \colon \dist(z,j) = \delta} \frac{R(F_0,z)}{R(F_0,j)}
  \]
  is a strictly monotonically increasing bijection.
\end{lemma}

\begin{proof}
  We know that the unique minimum of~$R(F_0,z)$ is attained when $z = z(F_0) = j$.
  We now fix $\phi$ in Lemma~\ref{L:Rquot} and consider $R(F_0,z)/R(F_0,j)$ as a function
  of~$\delta$. Note that
  \[ \cosh \delta + \langle \phi, \phi_k \rangle \sinh \delta
       = \frac{1 + \langle \phi, \phi_k \rangle}{2} e^\delta
          + \frac{1 - \langle \phi, \phi_k \rangle}{2} e^{-\delta}
  \]
  is a nonnegative linear combination of $e^\delta$ and~$e^{-\delta}$. This implies
  that $R(F_0,z)/R(F_0,j)$, which is the product of these expressions over all
  $1 \le k \le n$, is a nonnegative linear combination of
  terms~$e^{m\delta}$ (with $-n \le m \le n$), which are all convex from below.
  For $m \neq 0$, they are even strictly convex. The only possibility that
  results in a constant product is that half of the~$\phi_k$ equal~$\phi$ and
  the other half equal~$-\phi$, but this would imply that $F_0$ is not stable.
  So the quotient, considered as a function of~$\delta$ for fixed~$\phi$,
  is strictly convex from below with minimum value~$1$ at $\delta = 0$.
  In particular, $\delta \mapsto R(F_0,z)/R(F_0,j)$ is strictly increasing
  as $\delta$ grows from~$0$ to~$\infty$. It follows that $\eps_{F_0}$ is strictly increasing
  as well; also $\eps_{F_0}(0) = 1$. It remains to show that $\eps_{F_0}(\delta)$ tends
  to infinity with~$\delta$. Note that
  \[ \cosh \delta + \langle \phi, \phi_k \rangle \sinh \delta
         \ge \max \Bigl\{\frac{1 + \langle \phi, \phi_k \rangle}{2} e^\delta, e^{-\delta}\Bigr\} \,.
  \]
  Since $F_0$ is stable by assumption, the maximal multiplicity~$m$ of a root of~$F_0$ is
  strictly less than~$n/2$. Defining $\eta$ similarly as in the proof of Proposition~\ref{P:eps},
  but restricting to pairs $k, k'$ with $\phi_k \neq \phi_{k'}$, we deduce that
  \[ \frac{R(F_0,z)}{R(F_0,j)} \ge \eta^{n-m} e^{(n-2m) \delta} \ge \eta^{n-m} e^\delta \,, \]
  which finishes the proof.
\end{proof}

\begin{definition}
  Let $F \in \C[x,y]_n$ be stable. Let $F_0$ be a form in the
  $\SL(2, \C)$-orbit of~$F$ satisfying $z(F_0) = j$. Then we define
  $\eps_F(\delta) = \eps_{F_0}(\delta)$ with $\eps_{F_0}$ as in Lemma~\ref{L:epsdelta}.
  If $F$ is squarefree, then we define
  $\eps(F) = \eps(F_0)$ with $\eps(F_0)$ as in~\eqref{E:eps}.
\end{definition}

We remark that $\eps(F_0 \cdot \gamma) = \eps(F_0)$ for $\gamma \in \SU(2)$
(such $\gamma$ induce rotations of~$\Sph^2$ and so do not change the
geometry of the situation), which implies that $\eps_F$ (or $\eps(F)$) does not depend
on the choice of~$F_0$. When working over~$\R$, we take $F_0$ in the
$\SL(2, \R)$-orbit of~$F$; the previous remark then applies with $\SO(2)$ in
place of~$\SU(2)$.

We now combine the results obtained so far.

\begin{theorem} \label{T:bound}
  Let $F \in \C[x,y]_n$ be stable; we write $\delta = \dist(z(F), j)$.
  Then
  \[ 2^{1-n} \eps_F(\delta)
      \le \frac{\|F\|}{\theta(F)}
      \le 2^{-n} \binom{2n}{n} (\cosh \delta)^n \,.
  \]
  For $\gamma \in \SL(2,\C)$ write $\gamma^{-1} \cdot z(F) = t + uj$.
  If
  \[ \frac{|t|^2 + u^2 + 1}{2u}
       > \cosh \eps_F^{-1}\Bigl(2^{n-1} \frac{\|F\|}{\theta(F)}\Bigr) \,,
  \]
  then $\|F \cdot \gamma\| > \|F\|$.

  If $F$ is squarefree, then we can replace the condition on~$\gamma$ by
  \[ \frac{|t|^2 + u^2 + 1}{2u}
       > 2 \left(\frac{2 \|F\|}{\eps(F) \theta(F)}\right)^{1/(n-2)} \,.
  \]
\end{theorem}

\begin{proof}
  Recall that $\theta(F) = R(F, z(F))$. Choose $\gamma_0 \in \SL(2,\C)$ such that
  $\gamma_0 \cdot z(F) = j$; then $F_0 = F \cdot \gamma_0^{-1}$ satisfies
  $z(F_0) = j$ and $\eps(F) = \eps(F_0)$. By the invariance of~$R$,
  \[ \frac{\|F\|}{\theta(F)}
       = \frac{\|F\|}{R(F, j)} \cdot \frac{R(F, j)}{R(F, z(F))}
       = \frac{\|F\|}{R(F, j)} \cdot \frac{R(F_0, \gamma_0 \cdot j)}{R(F_0, j)} \,.
  \]
  Combining the definition of~$\eps_F$ with the bounds from
  Propositions \ref{P:bound1} and~\ref{P:eps}
  and using that $\dist(\gamma_0 \cdot j, j) = \delta$, we get the bounds
  on $\|F\|/\theta(F)$.
  For the second statement, we apply the lower bound for $F \cdot \gamma$
  in place of~$F$. Since $\theta(F \cdot \gamma) = \theta(F)$
  and $\eps_{F \cdot \gamma} = \eps_F$, this results in
  \[ \|F \cdot \gamma\|
        \ge 2^{1-n} \eps_F \bigl(\dist(\gamma^{-1} \cdot z(F), j)\bigr) \theta(F) \,.
  \]
  By~\eqref{E:coshdelta}, $\cosh \dist(\gamma^{-1} \cdot z(F), j) = (|t|^2 + u^2 + 1)/(2u)$,
  so the stated condition is equivalent to the right hand side being $> \|F\|$.

  For squarefree~$F$, we use the estimate
  \[ \eps_F(\delta) \ge \eps(F) (\cosh \delta)^{n-2} \,, \]
  which implies that any $\gamma$ satisfying the last condition also satisfies
  the previous one.
\end{proof}

The last part of the proof shows that using $\eps_F$ will in general
result in better bounds than using~$\eps(F)$. On the other hand, inverting $\eps_F$
may be algorithmically more involved.

Note that for the second statement, we only need the lower bound,
based on the lower bound in Proposition~\ref{P:bound1}, which is sharp for
some~$F$.

Now consider a squarefree (or stable) $F \in \R[x,y]_n$.
Then Theorem~\ref{T:bound} gives us a finite subset of the $\SL(2,\Z)$-orbit
of~$F$ that is guaranteed to contain the representative of minimal size.
We will turn this into an algorithm in the next section.

\begin{remark}
  If we want to replace~$\Q$ by a number field~$K$ with ring of integers~$\calO_K$,
  then we would have to
  work with the (diagonal) action of~$\SL(2, \calO_K)$ on a product
  of upper half planes and spaces (one upper half plane for each real
  embedding of~$K$ and one upper half space for each pair of complex
  embeddings). For each embedding, we get a covariant point of $F \in K[x,y]'_n$
  in the corresponding half plane or space. By taking products,
  we can extend the definitions of~$\|F\|$ and~$\theta(F)$ to this
  situation, and we should be able to prove a version of Theorem~\ref{T:bound}
  that applies to it. The major problem, however, will be the enumeration
  of the points in the $\SL(2, \calO_K)$-orbit of $z(F)$ (which is now
  the tuple of covariant points associated to~$F$) that have bounded distance
  from~$(j,j,\ldots,j)$. To our knowledge, there are no good general
  algorithms for this so far. In some special cases (like imaginary quadratic
  fields of class number~$1$), an approach similar to that described in
  Section~\ref{S:algo} should be workable, however.
\end{remark}

\begin{remark} \label{R:gen}
  With applications beyond binary forms in mind, we note that we can generalize
  Theorem~\ref{T:bound} to the following situation. Assume we consider
  some kind of objects~$\Phi$ associated to~$\PP^1$ and that we have
  a $\PGL(2)$-equivariant way of associating to~$\Phi$ a binary form~$F$
  (up to scaling). The object~$\Phi$ will be given in terms of coordinates of
  points or coefficients of binary forms, so there is actually an action of~$\GL(2)$;
  we assume that the map $\Phi \mapsto F$ is in fact $\GL(2)$-equivariant.
  Let $s(\Phi)$ denote some measure of the ``size'' or ``height'' of~$\Phi$;
  we need the further assumption that there are constants $C, k > 0$ such
  that $\|F\| \le C s(\Phi)^k$. Since the coefficients of~$F$ will usually be given
  as homogeneous polynomials in the coordinates or coefficients describing
  our ``model'' of~$\Phi$, such a bound will be easily established.
  Then, for $\Phi$ defined over~$\C$,
  we can deduce in the same way as in the proof of Theorem~\ref{T:bound}
  that when $\gamma^{-1} \cdot z(F) = t + uj$ and
  \[  \frac{|t|^2 + u^2 + 1}{2u}
       > \cosh \eps_F^{-1} \Bigl(2^{n-1} \frac{C s(\Phi)^k}{\theta(F)}\Bigr) \,,
  \]
  we have that $s(\Phi \cdot \gamma) > s(\Phi)$. We will apply this
  to endomorphisms of~$\PP^1$ in Section~\ref{sect_ds}.

  Another application is when we want to use the maximum norm~$H_\infty(F)$ instead
  of~$\|F\|$ to measure how large $F$ is. We have that
  $\|F\| \le (n+1) H_\infty(F)^2$, which leads to the condition
  \[  \frac{|t|^2 + u^2 + 1}{2u}
       > \cosh \eps_F^{-1} \Bigl(2^{n-1} \frac{(n+1) H_\infty(F)^2}{\theta(F)}\Bigr)
  \]
  that guarantees that $H_\infty(F \cdot \gamma) > H_\infty(F)$.
\end{remark}


\section{The algorithm} \label{S:algo}

The basic outline of an algorithm implementing Theorem~\ref{T:bound}
is clear. We assume that $F(x,y) \in \R[x,y]_n$ is stable (so in particular, $n \ge 3$).
In practice, we may have to assume that $F$ is squarefree, since some
implementations require this to be able to compute $z(F)$ and~$\theta(F)$.

\begin{algo} \label{Algo:main} \strut\vspace{-1ex} 
  \begin{enumerate}[1.]
    \item[] \textbf{Input:} A stable binary form $F \in \R[x,y]_n$.
    \item[] \textbf{Output:} $\gamma \in \SL_2(\Z)$
            and a binary form $F' = F \cdot \gamma$ that is a smallest representative for $F$.
    \item Compute $z(F)$ and $\theta(F)$.
    \item Determine $\gamma_0 \in \SL(2,\Z)$ such that $\gamma_0^{-1} \cdot z(F) \in \calF$. \\
          Replace $z(F)$ by $\gamma_0^{-1} \cdot z(F)$ and $F$ by $F \cdot \gamma_0$.
    \item Compute an upper bound $c$ for $\cosh \eps_F^{-1}(2^{n-1} \|F\|/\theta(F))$.
    \item \label{I:enumgamma}
          Let $\Gamma$ be the set of all $\gamma \in \SL(2,\Z)$ such that
          $\cosh \dist(\gamma^{-1} \cdot z(F), j) \le c$.
    \item Determine $\gamma \in \Gamma$ with $\|F \cdot \gamma\|$ minimal
          and return $\gamma_0 \gamma$ and $F \cdot \gamma$.
  \end{enumerate}
\end{algo}

It is explained in~\cite{Cremona2} how one can compute~$z(F)$ and~$\theta(F)$;
implementations are available in Magma~\cite{Magma} and Sage~\cite{sage}.
Let $z(F) = t_0 + u_0 j$ (with $t_0 \in \R$, since $F$ has real coefficients).
To compute~$c$, we construct a suitable~$F_0$, for example,
by first replacing $F$ by $F_1(x, y) = F(x + t_0 y, y)$ and then setting
$F_0(x, y) = F_1(u_0^{1/2} x, u_0^{-1/2} y)$ (or $F_1(x, y/u_0)$; scaling does not
affect~$\eps_{F_0}$). From the roots of~$F_0$ we can deduce the collection
of $\phi_k \in \Sph^2$. Then we can use Lemma~\ref{L:Rquot} to evaluate~$\eps_{F_0}(\delta)$
for a collection of values~$\delta$; the monotonicity of~$\cosh \circ \eps_{F_0}$ then
gives us a suitable~$c$. Alternatively, when $F$ is squarefree,
we can use~\eqref{E:eps} to find
a lower bound~$\eps$ for~$\eps(F) = \eps(F_0)$ and set
\[ c = 2 \Bigl(\frac{2 \|F\|}{\eps \theta(F)}\Bigr)^{1/(n-2)} \,. \]%
(If $n = 3$, we can simply set $c = 4 \|F\|/\theta(F)$; see Lemma~\ref{L:n=3}.)
As already mentioned, using $\eps_F$ will give better bounds
and therefore lead to a smaller search space than using~$\eps(F)$, but
inverting~$\eps_F$ may require some additional work.

The least straight-forward step is step~\ref{I:enumgamma} above,
which comes down to enumerating all points in the $\SL(2,\Z)$-orbit
of~$z(F)$ whose hyperbolic distance from~$j$ does not exceed
$\operatorname{arcosh} c$. One possibility to deal with this is to
think backwards: given a point~$z$ in the orbit of~$z(F)$, where we assume that
$z(F) \in \calF$, we get
from~$z$ to~$z(F)$ by the usual ``shift-and-invert'' procedure
(add $m \in \Z$ to~$z$ so that $|\Re(z+m)| \le \tfrac{1}{2}$;
if $|z+m| < 1$, replace by $-1/(z+m)$ and repeat), where the shifts
decrease the distance from~$j$ and the inversions do not change it.
So we get all points in the orbit with bounded distance from~$j$
by constructing a rooted tree with nodes labeled by $z \in \HH$ and edges
labeled by $S$ (inversion), $T$ (shift by~$+1$) or $T^{-1}$
(shift by~$-1$), as follows.

\begin{enumerate}[1.]
  \item The root is $z(F)$. It has three children:
        $-1/z(F)$ (edge labeled~$S$), $z(F) + 1$ (edge labeled~$T$)
        and $z(F) - 1$ (edge labeled~$T^{-1}$).
  \item Let $z = t + uj$ be a non-root node. Let $\ell$ be the label
        of the edge connecting the node to its parent.
        If $(t^2 + u^2 + 1)/(2u) > c$, then remove this node. Otherwise:
        \begin{enumerate}[a.]
          \item If $\ell \neq S$ and $(|t|-1)^2 + u^2 \ge 1$,
                add a child $-1/z$ (edge labeled~$S$).
          \item If $\ell \neq T$, add a child $z-1$ (edge labeled~$T^{-1}$).
          \item If $\ell \neq T^{-1}$, add a child $z+1$ (edge labeled~$T$).
        \end{enumerate}
\end{enumerate}

The condition $(|t| - 1)^2 + u^2 \ge 1$ is equivalent to
$|\Re(-1/z)| \le \tfrac{1}{2}$; this means that $z$ can have resulted
from an inversion step in the shift-and-invert procedure.
If $z(F)$ is close to $\pm \frac{1}{2} + \frac{\sqrt{3}}{2} j$, then
using the condition as stated can lead to cycling in the algorithm.
To avoid this, we can use any representative point~$r$ in the translate of~$\calF$
under~$\gamma$, for example the points in the orbit of $2j$, to test the
condition. We then store $-1/r$, $r-1$ or~$r+1$ in addition to $-1/z$, $z-1$
or~$z+1$ in the child node.

Since we can update the bound~$c$ when we have found a new temporary
minimum of~$\|F \cdot \gamma\|$, the most efficient way to perform the
tree search is to order the nodes by their distance from~$j$ and
always expand the node closest to~$j$ that has not yet been dealt with.

Of course, we also keep track of $F \cdot \gamma$: a shift by~$\pm 1$
in~$z$ corresponds to substituting $(x \mp y,y)$ for~$(x,y)$,
and an inversion corresponds to substituting $(y, -x)$ for~$(x,y)$.

\begin{figure}[htb]
  \begin{center}
    \Gr{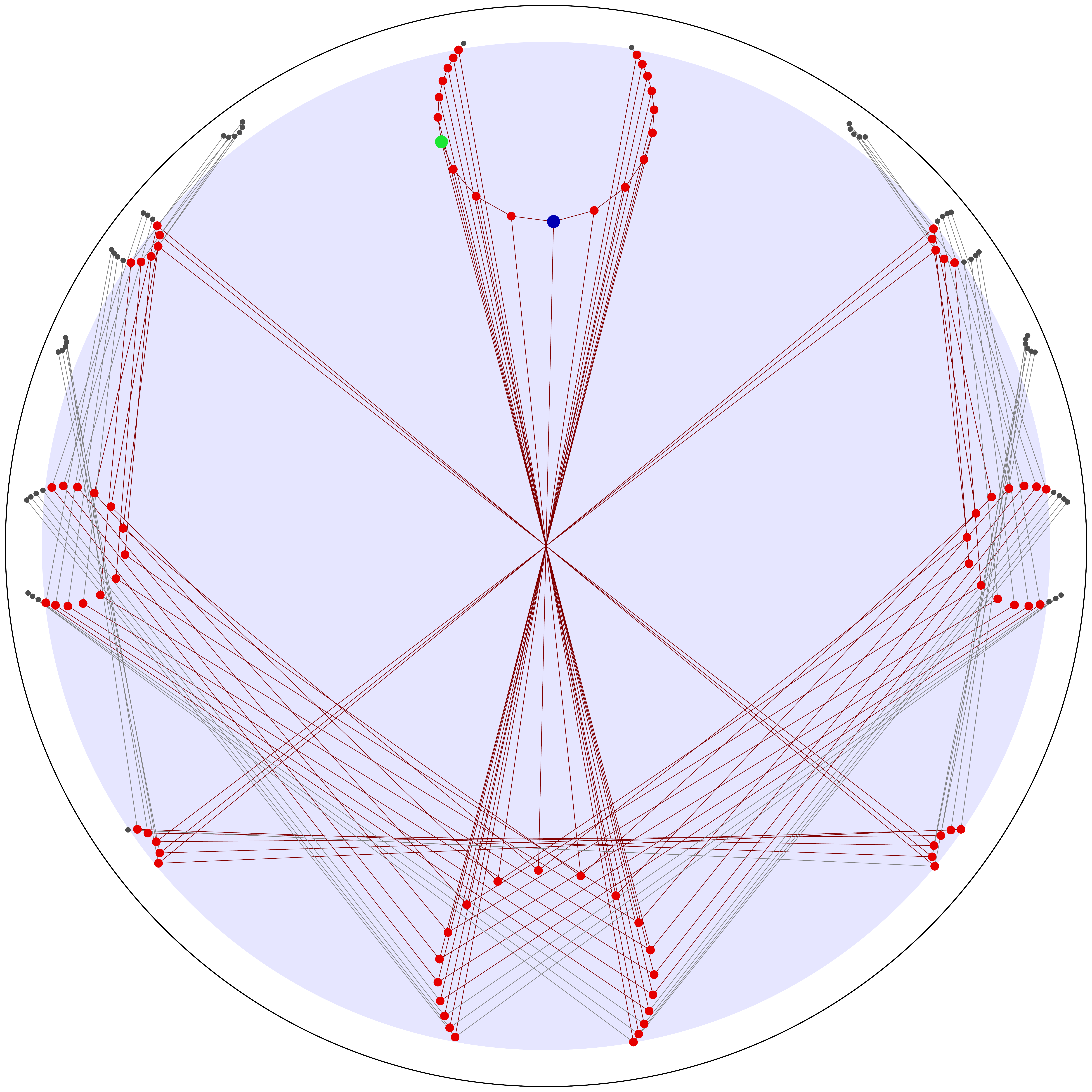}{0.7\textwidth} \\[10pt]
    \caption{The tree that is traversed during the computation of the
             representative of minimal size in Example~\ref{Ex1}, in the
             Poincar\'e disk model. The light blue disk bounds the search
             region. The larger blue dot represents~$z(F)$ and the larger
             green dot is where the minimum is attained. The red dots
             mark the nodes that are expanded; the gray dots are nodes
             that are discarded.}
    \label{Fig1}
  \end{center}
\end{figure}

\begin{example} \label{Ex1}
  Consider the cubic form
  \[ F(x,y) = -2 x^3 + 2 x^2 y + 3 x y^2 + 127 y^3 \]
  with $\|F\| = 16\,146$. This form is reduced in the sense of~\cite{Cremona2}, since
  \[ z(F) \approx 0.17501 + 3.99543 j \in \calF \,. \]
  We then need to find the $\gamma \in \SL(2,\Z)$ with $\gamma^{-1} \cdot z(F) = t + uj$ satisfying
  \[ \frac{t^2 + u^2 + 1}{2u} \le c
        = \cosh \eps_F^{-1} \Bigl(\frac{4 \|F\|}{\theta(F)}\Bigr)
        \approx 28.0049 \,.
  \]
  For comparison, if we use $\eps(F) = 1$ instead, the initial bound we
  obtain is $c\approx 31.5022$.
  We use the tree search explained above. In the course of the search, the bound
  gets reduced to~$\approx 14.3415$; there are then $88$~points
  (marked red in Figure~\ref{Fig1}) in the $\SL(2,\Z)$-orbit of~$z(F)$ that have to be considered.
  A representative of smallest size is obtained for $\gamma = \smm{1}{4}{0}{1}$ with
  \[ \|F \cdot \gamma\| = \|F(x+4y,y)\|
        = \|{-2} x^3 - 22 x^2 y - 77 x y^2 + 43 y^3\| = 8\,266 \,.
  \]
  If instead we minimize the height, then the bound~$c$ is $\approx 111.891$,
  which gets reduced to~$\approx 23.3403$ during the search.
  There are $140$~points in the search region, and a smallest representative is
  \[ F(4x-5y, x-y) = 43 x^3 - 52 x^2 y - 47 x y^2 + 58 y^3 \qquad\text{of height~$58$} \]
  (compared to $H_\infty(F) = 127$). The covariant point of this form
  is $\approx 1.12502 + 0.13060 j$; the $\cosh$ of its distance to~$j$
  is $\approx 8.73973$.
\end{example}


\section{Application to dynamical systems} \label{sect_ds}

A \emph{dynamical system on~$\PP^1$} is a non-constant
endomorphism $f \colon \PP^1 \to \PP^1$. As described in Section~\ref{S:endo},
$f$ can be specified by a model $[F : G]$, where $F$ and~$G$ are binary
forms of the same degree~$d$ and without common factors. If $f$ is defined over~$\Q$,
we can choose $F, G \in \Z[x,y]_d$ with coprime coefficients.
We have a natural right action of~$\SL(2, \Z)$ on endomorphisms of degree~$d$
by conjugation, which is given explicitly for
$\gamma = \smm{a}{b}{c}{d} \in \SL(2,\Z)$ by
\[ f^\gamma = [F^\gamma : G^\gamma]\,, \]
where
\begin{align*}
  F^\gamma(x,y) &= d F(ax+by, cx+dy) - b G(ax+by, cx+dy) \qquad\text{and} \\
  G^\gamma(x,y) &= -c F(ax+by, cx+dy) + a G(ax+by, cx+dy) \,.
\end{align*}
This action can be extended to $\Mat(2,\Z) \cap \GL(2,\Q)$.
Note that this amounts to using the
adjoint~$\smm{d}{-b}{-c}{a}$ in place of the inverse
compared to the action by conjugation. Since scaling
both forms by a common factor does not change the endomorphism they
represent, this still induces the usual action of~$\PGL(2,\Q)$
on endomorphisms by conjugation.

To apply the reduction algorithm of binary forms to dynamical systems,
we follow the framework of Remark~\ref{R:gen}.
We first associate to each dynamical system a binary form in a covariant way.
Let $f \colon \PP^1 \to \PP^1$ be a dynamical system defined over~$\Q$.
Choose a model $f = [F : G]$ with two homogeneous polynomials $F$ and~$G$
with coprime integral coefficients.
We write the $m$th iterate of~$f$ as $f^m = [F_m : G_m]$. Define
\[ \Phi_m(f) = y F_m - x G_m \qquad\text{and}\qquad
   \Phi_m^{\ast}(f) = \prod_{k \mid m} (y F_k - x G_k)^{\mu(m/k)} \,,
\]
where $\mu$ is the M\"obius function. The zeros of the form~$\Phi_m(f)$
are the points of period~$m$ for~$f$. The form $\Phi^{\ast}_m(f)$
is called the \emph{dynatomic polynomial} and its zeros, in most cases,
are the points of minimal period~$m$ for~$f$.
See~\cite{Silverman10}*{Section~4.1}
for properties of dynatomic polynomials in dimension~$1$.
It is easy to check that $\Phi_m(f^\gamma) = \Phi_m(f) \cdot \gamma$
and similarly for~$\Phi^*_m$.

We can bound the size of~$\Phi_m(f)$ and~$\Phi^{\ast}_m(f)$ in terms of
the height of~$f = [F : G]$.

\begin{proposition}\label{prop4}
  Given $d \ge 2$ and $m \ge 1$,
  there exist positive constants $C_{d,m}, C_{d,m}^{\ast}$ and
  $k_{d,m}, k_{d,m}^{\ast}$ such that for every morphism
  $f \colon \PP^1 \to \PP^1$ of degree $d \geq 2$ as above,
  \[ \|\Phi_m(f)\| \leq C_{d,m} H(f)^{k_{d,m}} \qquad\text{and}\qquad
     \|\Phi_m^{\ast}(f)\| \leq C_{d,m}^{\ast} H(f)^{k^{\ast}_{d,m}}\,.
  \]
\end{proposition}

\begin{proof}
  The coefficients of $\Phi_m(f)$ and~$\Phi^{\ast}_m(f)$ are homogeneous polynomials
  in the coefficients of~$f$. Then the sum~$s(f)$ of the squares of the coefficients of~$\Phi_m(f)$
  or~$\Phi^\ast_m(f)$ is a homogeneous polynomial of some degree $k_{d,m}$ or~$k^\ast_{d,m}$
  in the coefficients of~$f$. The bound on $\|\Phi^{(\ast)}_m(f)\| = s(f)$ is obtained by bounding
  the coefficients of~$f$ by~$H(f)$ and applying the triangle inequality.
\end{proof}

For given $d$ and~$m$, we can find suitable constants (at least in principle)
by doing the computation mentioned in the proof for a generic~$f$. This gives
\begin{align*}
  C_{d,1} &= C^\ast_{d,1} = 4d + 2\,, & k_{d,1} &= k^\ast_{d,1} = 2\,;  &
  C_{2,2} &= 322\,,                   & k_{2,2} &= 6\,; \\
  C^\ast_{2,2} &= 43\,,               & k^\ast_{2,2} &= 4\,;  &
  C^\ast_{2,3} &= 106\,459\,,         & k^\ast_{2,3} &= 12\,; \\
  C_{3,2} &= 18\,044\,,               & k_{3,2} &= 8\,;  &
  C^\ast_{3,2} &= 1604\,,             & k^\ast_{3,2} &= 6\,.
  \renewcommand\arraystretch{1.4}
\end{align*}

For most endomorphisms, using $m = 1$ is sufficient to get a squarefree
binary form of degree at least~$3$, but for some maps
we need to consider higher order periodic points.

Applying Remark~\ref{R:gen} now results in the following.

\begin{corollary} \label{C:endobound}
  Let $f = [F : G] \colon \PP^1 \to \PP^1$ be a morphism of degree~$d$, given
  by two binary forms $F, G \in \Z[x,y]_d$ with coprime coefficients.
  Pick some $\Phi = \Phi_m(f)$ or~$\Phi^\ast_m(f)$ such that $\Phi$
  is stable (in particular, $\deg \Phi \ge 3$).
  Let $C = C_{d,m}$ or~$C^\ast_{d,m}$ and $k = k_{d,m}$ or~$k^\ast_{d,m}$
  be the constants from Proposition~\ref{prop4},
  depending on the choice of~$\Phi$.

  Let $\gamma \in \SL(2, \Z)$ and write $\gamma^{-1} \cdot z(\Phi) = t + uj$. If
  \[ \frac{t^2 + u^2 + 1}{2 u}
      > \cosh \eps_\Phi^{-1} \left(2^{\deg \Phi - 1} \frac{C H(f)^k}{\theta(\Phi)}\right) \,,
  \]
  then $H(f^\gamma) > H(f)$.
\end{corollary}

This easily translates into an algorithm that produces a representative
of smallest height in a given $\SL(2,\Z)$-orbit.
We just have to use the bound from Corollary~\ref{C:endobound}
in place of~$c$ in the algorithm of Section~\ref{S:algo} and search
for the minimal~$H(f^\gamma)$ (reducing the bound whenever possible).
However, the problem remains of finding the representative of smallest height
over \emph{all} $\SL(2,\Z)$-orbits of minimal models.
Bruin and Molnar prove that for $f$ defined over~$\Q$,
there are only finitely many $\SL(2,\Z)$-orbits of minimal
models; see~\cite{Bruin3}*{Proposition~6.5}. So to obtain
a reduced model, we only have to apply the algorithm
of this paper to a representative from each $\SL(2,\Z)$-orbit
of minimal models. We will discuss in Section~\ref{S:orbits}
how we can obtain such representatives.

The following example shows how to find
the smallest height representative for an endomorphism of~$\PP^1$.

\begin{example} \label{exmp2}
  Consider the endomorphism
  \begin{align*}
      f \colon \qquad \PP^1 &\To \PP^1 \\
          (x:y) &\longmapsto (50 x^2 + 795 x y + 2120 y^2 : 265 x^2 + 106 y^2)\,.
  \end{align*}
  Since $f$ has even degree and the given model is minimal,
  we need only consider its $\SL(2,\Z)$-orbit;
  see Proposition~\ref{prop:even minimal}.

  The binary form defining its fixed points is
  \[ \Phi_1(f) = \Phi(x,y) = 265 x^3 - 50 x^2 y - 689 x y^2 - 2120 y^3 \,. \]
  The bound for the $\cosh$ of the distance~$\delta$ to~$j$ for finding the
  representative of smallest height in the orbit of~$\Phi$ is $\approx 14.2268$.
  As soon as the optimum is found, this bound is reduced to $\approx 9.6546$,
  which cuts down the number of points in the search space to~$54$.
  We find that
  \[ \Phi(x-y, y) = 265 x^3 - 845 x^2 y + 206 x y^2 - 1746 \]
  has both smallest height and smallest size in the $\SL(2,\Z)$-orbit of~$\Phi$.
  This corresponds to the endomorphism
  \[ (x : y) \longmapsto (-315 x^2 - 165 x y - 1746 y^2 : -265 x^2 + 530 x y - 371 y^2) \,. \]
  of height~$1746$. However, this is not the representative of smallest
  height among the $\SL(2,\Z)$-conjugates of~$f$: running the algorithm
  with the modifications indicated above, we get an initial bound
  of~$\approx 35.5547$ for~$\cosh \delta$, which gets reduced to~$\approx 19.7017$.
  The algorithm runs through $118$~points $\gamma^{-1} \cdot z(\Phi)$, showing
  that a representative of smallest height is obtained for the conjugate
  \[ f^\gamma
      = [480 x^2 + 1125 x y - 1578 y^2 : -265 x^2 - 1060 x y - 1166 y^2]
  \]
  of height~$1578$ for $\gamma = \smm{1}{2}{0}{1}$.
\end{example}


\section{Orbits of minimal models} \label{S:orbits}

Questions about the structure of the set of minimal models
of a dynamical system, including how to calculate a minimal model,
have been studied in two
different contexts. Bruin-Molnar~\cite{Bruin3}, in addressing questions
about integer points in orbits, consider the problem over number fields.
Rumely~\cite{Rumely} approaches the problem
over a complete algebraically closed nonarchimedean field and applies
Berkovich space methods to solve the problem.
We will use input from both publications
to devise an algorithm that finds representatives
of all $\GL(2,\Z)$-orbits of minimal models of a given endomorphism.
Furthermore, Rumely
shows that the valuation of the resultant factors through
a map from the Berkovich projective line~$\PBerk$ to~$\R$ and proves
that for even degree maps, the minimal valuation is achieved at a single point
in~$\PBerk$. However, this point does not necessarily correspond
to a model defined over~$\Q_p$. We will use his results to show that also
over~$\Q_p$, the minimal valuation of the resultant is obtained for a
unique $\GL(2,\Z_p)$-orbit of models; see Proposition~\ref{prop:even minimal} below.

To determine a representative from each orbit, we recall the key points of the
algorithm of Bruin-Molnar and then adapt it to our purposes.
Given a model $[F : G]$ of an endomorphism~$f$, where
$F, G \in \Q[x,y]_d$, we can scale $F$ and~$G$ by some $\lambda \in \Q^\times$,
which results in the model $[\lambda F, \lambda G]$ of the same~$f$,
and we can conjugate it by some element~$\gamma$ of~$\GL(2, \Q)$.
In section~\ref{sect_ds}, we defined an action of $\gamma \in \GL(2,\Q)$ on~$f$
that avoids introducing denominators. As in Bruin-Molnar, these two operations
combine to act on~$[F : G]$ by~$(\lambda, \gamma = \smm{a}{b}{c}{d})$
resulting in
\[ [\lambda F^\gamma : \lambda G^\gamma]
      = \bigl[\lambda \bigl(d F(ax+by, cx+dy) - b G(ax+by, cx+dy)\bigr) :
              \lambda \bigl(-c F(\ldots) + a G(\ldots)\bigr)\bigr] \,.
\]
Acting by~$(\lambda^{-d-1}, \lambda I_2)$ has no effect (here $I_2$ denotes
the $2 \times 2$ identity matrix), so we can assume $\gamma$ to have integral entries.
Recall that we are interested in \emph{minimal} models of~$f$, which are
integral models with minimal absolute value of the resultant.
According to~\cite{Bruin3}*{Proposition 2.2}, we have that
\begin{equation} \label{E:change}
    \Res(\lambda F^{\gamma}, \lambda G^{\gamma}) = \lambda^{2d}\det(\gamma)^{d^2+d}\Res(F,G)\,.
\end{equation}
Each prime can be considered separately, and
\cite{Bruin3}*{Proposition~6.3} shows that we need only
consider primes that divide the resultant of the given integral model.
We can therefore make the resultant smaller if, for a prime~$p$ dividing the resultant,
we can find
$\gamma \in \Mat(2,\Z) \cap \GL(2,\Q)$ such that the gcd of the coefficients
of $F^\gamma$ and~$G^\gamma$ is~$\alpha$ with
$2v_p(\alpha) > (d+1)v_p(\det(\gamma))$, where $v_p(\cdot)$ is the
(normalized) $p$-adic valuation.

Further, \cite{Bruin3}*{Proposition~2.12} shows that it is sufficient to consider affine
transformations. Let $p$ be a prime dividing the resultant and
consider the affine transformation and scaling factor
\begin{equation*}
    \gamma = \begin{pmatrix} p^{e_2} & \beta \\ 0 & 1 \end{pmatrix} \qquad \text{and} \qquad
    \lambda = p^{-e_1}\,.
\end{equation*}

In this form, the condition for the power of~$p$ dividing the resultant to decrease then becomes
\begin{equation} \label{eq1}
    2 e_1 > (d+1) e_2 \,,
\end{equation}
where $e_1$ is the minimal $p$-adic valuation of a coefficient
of
\[ F^\gamma = F(p^{e_2} x + \beta y, y) - \beta G(p^{e_2} x + \beta y, y) \quad\text{or}\quad
   G^\gamma = p^{e_2} G(p^{e_2} x + \beta y, y) \,.
\]

We now make use of Rumely's results from~\cite{Rumely}.
We write~$\PBerk$ for the Berkovich projective line over~$\C_p$, the completion
of the algebraic closure of~$\Q_p$, and we denote the ``Berkovich upper half space''
$\PBerk \setminus \PP^1(\C_p)$ by~$\HBerk$.
Fix an endomorphism~$f$ of~$\PP^1$ over~$\C_p$.
Rumely shows that there is a continuous, piecewise affine with integral slopes
(with respect to the logarithmic distance on~$\HBerk$) and convex map
$\ordRes_f \colon \HBerk \to \R_{\ge 0}$ such that
$v_p(\Res(f^\gamma)) = \ordRes_f(\gamma \cdot \zeta_G)$ for all
$\gamma \in \GL(2,\C_p)$, where $\zeta_G \in \PBerk$ is the Gauss point.
Here, $\Res(f^\gamma)$ denotes the resultant of a representative $[F : G]$ of~$f^\gamma$
that is scaled so that $F$ and~$G$ have $p$-adically integral coefficients with one
of them a unit (i.e., we take the maximal~$e_1$ in the notation above).
The orbit of~$\zeta_G$ under~$\GL(2,\Q_p)$ consists of the
set~$\calV$ of vertices of a subtree~$\calT$ of~$\HBerk$ whose edges have length~$1$
in the logarithmic metric; the vertices have degree~$p+1$.

The $\GL(2,\Z_p)$-orbits of minimal models of~$f$ then correspond to the
points in~$\calV$ in which $\ordRes_f|_\calV$ takes its minimal value.
Note that it is possible that $\ordRes_f$ will take on a smaller value at a
point of $\calT$ which is not a vertex, corresponding to a model
defined over an extension of $\Q_p$,
see Example~\ref{E:min_res}. The restriction of $\ordRes_f$ to~$\calT$
is still piecewise affine and convex. This implies that a point in~$\calV$
is a minimizer of~$\ordRes_f|_\calV$ if (and only if) it is a local minimum, in the sense that
$\ordRes_f$ does not take a strictly smaller value at a neighboring vertex.
It also implies that at each vertex, there is at most one edge
leading to a vertex with strictly smaller value, and following these edges leads
to a minimum. As noted by Rumely, this can be used to simplify the Bruin-Molnar algorithm.
We obtain the following procedure for finding a $p$-adically minimal model.

We assume that a \emph{normalized} model $f = [F : G]$ is given, i.e.,
such that $F, G \in \Z_p[x,y]_d$, where at least one coefficient is a $p$-adic unit.

\begin{algo} \label{Algo:min} \strut\vspace{-1ex} 
  \begin{enumerate}[1.]
    \item[] \textbf{Input:} A normalized model $f = [F:G]$ of a dynamical system
            of degree~$d$ over~$\Z_p$.
    \item[] \textbf{Output:} $\gamma_0 \in \GL(2, \Q_p)$
             and $f' = f^{\gamma_0}$ such that $f'$ is a $p$-adically minimal model for $f$.
    \item Let $T = \bigl\{\smm{1}{0}{0}{p}\bigr\}
                      \cup \bigl\{\smm{p\mathstrut}{a}{0}{1} : a \in \Z,\; 0 \le a < p\bigr\}$.
          Set $\gamma_0 = \smm{1}{0}{0}{1}$.
    \item \label{step2}
          For $\gamma \in T$, compute $f^\gamma = [\lambda F^\gamma : \lambda G^\gamma]$,
          where $\lambda$ is chosen so that the resulting model is normalized. \\
          If we come from step~\ref{step3}, then we leave out the one~$\gamma$
          that would bring us back to the $\GL(2,\Z_p)$-orbit of the previously considered model.
    \item \label{step3}
          If $v_p(\Res(f^\gamma)) < v_p(\Res(f))$ for some $\gamma$,
          then replace $f$ with~$f^\gamma$ and $\gamma_0$ with $\gamma_0 \gamma$.\\
          If $v_p(\Res(f^\gamma)) \geq d$ for $d$ even or $\geq 2d$ for $d$ odd, then go to step~\ref{step2}.
    \item Otherwise, return~$f$ and~$\gamma_0$.
  \end{enumerate}
\end{algo}

The set~$T$ contains representatives~$\gamma$ that map $\zeta_G$
to each of its $p+1$~neighbors in~$\calT$.
If we have used $\gamma = \smm{p}{a}{0}{1}$ to reduce the valuation of the resultant,
then we exclude $\smm{1}{0}{0}{p}$ when we carry out step~\ref{step2} the next time;
if we have used $\gamma = \smm{1}{0}{0}{p}$, then we exclude $\smm{p}{0}{0}{1}$.
Note that from~\eqref{E:change}, we can deduce that the difference
of the values of~$\ordRes_f$ at neighboring vertices of~$\calT$ is divisible
by~$d$ when $d$ is even and by~$2d$ when $d$ is odd. So if the value for the
model we are considering is smaller than $d$ or~$2d$, respectively, then the model
must be minimal.

This algorithm is essentially an enumerative approach similar to both the algorithms
in Rumely~\cite{Rumely} and Bruin-Molnar~\cite{Bruin3}. In the case of Rumely, there
is additional logic to limit the number of directions to consider and to compute how
far to move in each direction. In the case of Bruin-Molnar, a set of inequalities
is solved to determine the $a$ to use in step~\ref{step2}. For reasonably small primes,
this enumerative approach is sufficient. For larger primes, a version
of the Bruin-Molnar inequality solver can be used to speed up steps \ref{step2} and~\ref{step3}.

If the degree~$d$ of~$f$ is even, then \eqref{E:change} shows that the change of~$\ordRes_f$
along each edge is congruent to~$d^2+d$ modulo $2d$ and so is never zero.
Consequently, if we have found a
minimizing vertex, then $\ordRes_f$ will strictly increase along all edges emanating
from it, which implies that this vertex is the unique minimizer. Since this holds
for each prime~$p$, we obtain a proof of the following
statement, which answers Question~6.2 in~\cite{Bruin3} in the
affirmative in the strongest possible sense.
(The argument is already in~\cite{Rumely}*{p.~280}, if somewhat implicit.)

\begin{proposition}\label{prop:even minimal}
  Let $f \colon \PP^1 \to \PP^1$ be defined over~$\Q$.
  If the degree of $f$ is even, then $f$ has
  a single $\GL(2,\Z)$-orbit of minimal models.
\end{proposition}

It is possible that the minimal value of~$\ordRes_f$ on~$\HBerk$
is not attained on~$\calT$; see Example~6.2 in~\cite{Rumely}. The following
example shows that it is also possible that the minimum is attained on~$\calT$,
but not on~$\calV$.

\begin{example}\label{E:min_res}
    Looking more closely at Example~6.1 from Rumely~\cite{Rumely},
    we consider the endomorphism
    \begin{align*}
        f \colon \qquad \PP^1 &\To \PP^1 \\
            (x:y) &\longmapsto [x^2 - p y^2 : x y]\,,
    \end{align*}
    where $p$ is a prime. In this case, the minimum value of $\ordRes_f$
    over the algebraic closure occurs between two vertices of~$\calT$.
    Specifically, $\Res(f) =  -p$ is minimal over~$\GL(2,\Q_p)$, but going
    to the ramified quadratic extension $\Q_p(\sqrt{p})$,
    for $\gamma = \smm{\sqrt{p}}{0}{0}{1}$
    we have, after normalizing, $\Res(f^{\gamma}) = -1$. The next vertex of~$\calT$
    that lies in this direction corresponds to conjugating by $\alpha = \smm{p}{0}{0}{1}$
    and has $\Res(f^{\alpha}) = -p^3$. In particular, the function~$\ordRes_f$ has
    slope $-d = -2$ for half
    of the edge containing the minimum~$f^{\alpha}$
    and $3d = 6$ for the remaining half,
    giving a net increase of~$d = 2$ along the edge.
\end{example}

Rumely shows that when the degree~$d$ is odd, $\ordRes_f$ achieves
its minimum value at a unique point or along an interval in~$\HBerk$.
If this set meets~$\calV$, then the intersection is the set
of minimizers of~$\ordRes_f|_\calV$, and this set is either one point or
consists of the vertices in a path in~$\calT$. If the minimizing set
meets~$\calT$, but not~$\calV$, then the intersection with~$\calT$ must be contained in the interior
of an edge of~$\calT$. Otherwise, Proposition~3.5
of~\cite{Rumely} implies that $\ordRes_f|_\calT$ has a unique minimum,
which is attained at an interior point of an edge.
In both of these last two cases, the set of minimizers
of~$\ordRes_f|_\calV$ consists either of one or both of the endpoints of this edge.
(We can indeed have two minimizing vertices in this case;
this is what happens in Example~6.4 in~\cite{Rumely}.)
So in all cases, the subset of~$\calV$ corresponding to $\GL(2,\Z_p)$-orbits
of $p$-minimal models of~$f$ is either one point or consists of the vertices
in a path. This path can have any length. This is demonstrated
by the following example, which extends Example~6.1 of~\cite{Bruin3}.

\begin{proposition}\label{prop:odd minimal}
    Let $0 \neq c \in \Z$ and $n$ a positive integer. Define
    \begin{align*}
        f  \colon \qquad \PP^1 &\To \PP^1 \\
        (x : y) &\longmapsto [x^{2n+1} - c^{n+1} y^{2n+1} : x^n y^{n+1}].
    \end{align*}
    Let $c = \pm \prod_{i=1}^s p_i^{e_i}$
    be the prime factorization of~$c$. Then
    $f$ has exactly $\prod_{i=1}^s (e_i+1)$ distinct $\GL(2,\Z)$-orbits of minimal
    models, one for each positive divisor of~$c$.
\end{proposition}

\begin{proof}
  Write $c = r s$ with $r, s \in \Z$. Acting on the given model
  by $(\lambda, \gamma) = \bigl(r^{-n-1}, \smm{r}{0}{0}{1}\bigr)$,
  we obtain the minimal model $[r^n x^{2n+1} - s^{n+1} y^{2n+1} : x^n y^{n+1}]$.
  The models associated to the factorizations $c = r s$ and $c = r' s'$
  are in the same $\GL(2,\Z)$-orbit if and only if the corresponding
  matrices differ multiplicatively by a scalar multiple of a matrix in~$\GL(2,\Z)$,
  which is the case if and only if $r' = \pm r$. So we obtain a distinct $\GL(2,\Z)$-orbit
  of minimal models for each positive divisor~$r$ of~$c$.

  To see that these models cover all $\GL(2,\Z)$-orbits, we consider
  an affine transformation \hbox{$\gamma = \smm{\alpha}{\beta}{0}{\delta}$}
  with $\alpha, \beta, \delta \in \Z$ coprime and $\alpha, \delta > 0$.
  Since we are interested in orbits under~$\GL(2,\Z)$, we can assume that $|\beta| < \alpha$.
  Denoting the polynomials in the given model by $F$ and~$G$, we have that
  \[ F^\gamma = \delta (\alpha x + \beta y)^{2n+1} - \delta^{2n+2} c^{n+1} y^{2n+1}
                  - \beta (\alpha x + \beta y)^n \delta^{n+1} y^{n+1}
  \]
  and
  \[ G^\gamma = \alpha (\alpha x + \beta y)^n \delta^{n+1} y^{n+1} \,. \]
  If this is to lead to a minimal model, $\det(\gamma)^{n+1} = \alpha^{n+1} \delta^{n+1}$
  must divide all coefficients of $F^\gamma$ and~$G^\gamma$. Considering the
  $y^{2n+1}$~term in~$G^\gamma$, we see that $\alpha$ must divide~$\beta$,
  so that $\beta = 0$. Then $\alpha$ and~$\delta$ are coprime, and
  \[ F^\gamma = \alpha^{2n+1} \delta x^{2n+1} - \delta^{2n+2} c^{n+1} y^{2n+1} \,, \]
  so $\delta$ divides~$\alpha$ and $\alpha$ divides~$c$. Since $\alpha$ and~$\delta$
  are coprime, this means that $\delta = 1$ and $\alpha = r$ with $c = rs$ a factorization
  as above with $r > 0$.
\end{proof}

\begin{remark}
    Note that $c = \pm 1$ results in a map with a single orbit, so it is
    possible for odd degree maps to have a single orbit of minimal models.
\end{remark}

To get representatives of \emph{all} $\GL(2,\Z_p)$-orbits of minimal models
of a given dynamical system~$f = [F : G]$, we modify the algorithm
in the following way. Traverse the vertices of $\calT$ until the minimal value
of~$\ordRes_f$ is attained. Then, find all vertices of~$\calT$ that attain that minimum;
these will all lie on a path in~$\calT$, so can be determined one edge at a time.
Note that the first minimal model we find could correspond to a
vertex in the interior of this path, so we may have to search in two directions
to find all vertices in the path. The algorithm below returns a set of pairs $(f^\gamma, \gamma)$
with $\gamma \in \GL(2,\Q_p)$ such that the~$f^\gamma$ represent all $\GL(2,\Z_p)$-conjugacy
classes of minimal models of~$f$.

\begin{algo} \strut\vspace{-1ex} 
  \begin{enumerate}[1.]
    \item[] \textbf{Input:} A normalized model $f = [F:G]$ of a dynamical system
            of degree~$d$ over~$\Z_p$.
    \item[] \textbf{Output:} A list of pairs $(f',\gamma)$ so that $f' = f^{\gamma}$
            is a minimal model for $f$, \\
            \strut\hphantom{\textbf{Output:}} each representing a distinct $\GL(2,\Z_p)$-orbit.
    \item Let $T = \bigl\{\smm{1}{0}{0}{p}\bigr\}
                      \cup \bigl\{\smm{p\mathstrut}{a}{0}{1} : a \in \Z,\; 0 \le a < p\bigr\}$.
    \item Use Algorithm~\ref{Algo:min} to find a minimal model~$f_0 = f^{\gamma_0}$ of~$f$. \\
          Set $M \leftarrow \{(f_0, \gamma_0)\}$ and $f \leftarrow f_0$.
    \item Determine $S = \{\gamma \in T : v_p(\Res(f^\gamma)) = v_p(\Res(f))\}$.
    \item For each element $\gamma \in S$ (there are at most two),
          call $\text{\textbf{search}($f^\gamma$, $\gamma_0 \gamma$, $\gamma$)}$.
    \item Return $M$.
  \end{enumerate}
  \textbf{search}($f$, $\gamma_0$, $\gamma$):\vspace{-1ex}
  \begin{enumerate}[a.]
    \item \label{stepa}
          Set $M \leftarrow M \cup \{(f, \gamma_0)\}$.
    \item Determine $S = \{\gamma' \in T : \gamma \gamma' \notin p \GL(2,\Z_p),
                                            v_p(\Res(f^{\gamma'})) = v_p(\Res(f))\}$.
    \item If $S = \{\gamma'\}$, then set
          $(f, \gamma_0, \gamma) \leftarrow (f^{\gamma'}, \gamma_0 \gamma', \gamma')$;
          go to step~\ref{stepa}.
    \item Otherwise, $S = \emptyset$. Return.
  \end{enumerate}
\end{algo}

In the call to~\textbf{search}, the last argument~$\gamma$ specifies
the direction we come from. See the discussion after Algorithm~\ref{Algo:min}
for what the condition $\gamma \gamma' \notin p \GL(2,\Z_p)$ amounts to
in terms of directions in Berkovich space.

To get representatives of all $\GL(2\,\Z)$-orbits of minimal
models, we run the above algorithm for each prime dividing the resultant of the
given model, applying the returned~$\gamma_0$ to each of the models obtained so far.

We finally note that any $\GL(2,\Z)$-orbit splits into at most two
$\SL(2,\Z)$-orbits and that the action of $(1, \smm{-1}{0}{0}{1})$
preserves the height of the model. It is therefore sufficient to look
at the $\SL(2,\Z)$-orbits of the representatives of the $\GL(2,\Z)$-orbits
when we want to find a reduced model.


\end{document}